\documentclass[a4paper, 11pt]{amsart}
\usepackage{bm,amsfonts,amsmath,amssymb,bbm,mathrsfs,amsthm}

\usepackage{graphicx}
\usepackage{soul}

\usepackage{color,hyperref,eurosym, eucal,palatino, marginnote}

\definecolor{darkblue}{rgb}{0.2,0.2,0.6}
\definecolor{darkblue2}{rgb}{0.2,0.2,0.9}
\definecolor{superdarkblue}{rgb}{0.2,0.2,0.3}
\hypersetup{colorlinks=true, linkcolor=darkblue,citecolor=darkblue,
urlcolor = superdarkblue}
\definecolor{citegreen}{rgb}{0.2,0.2,0.6}

\usepackage[normalem]{ulem}
\definecolor{DarkGreen}{rgb}{0,0.5,0.1}
\definecolor{DarkYellow}{rgb}{1,0.7,0} 

\newcommand\soutD{\bgroup\markoverwith
{\textcolor{DarkYellow}{\rule[.5ex]{2pt}{1pt}}}\ULon}
\newcommand\soutP{\bgroup\markoverwith
{\textcolor{blue}{\rule[.5ex]{2pt}{1pt}}}\ULon}
\newcommand{\Hm}[1]{\leavevmode{\marginpar{\tiny%
$\hbox to 0mm{\hspace*{-0.5mm}$\leftarrow$\hss}%
\vcenter{\vrule depth 0.1mm height 0.1mm width \the\marginparwidth}%
\hbox to
0mm{\hss$\rightarrow$\hspace*{-0.5mm}}$\\\relax\raggedright #1}}}

\usepackage{geometry}

\newcommand\Op{\sfH_\mu}
\newcommand\OpD{\sfH_{\rm D}}

\newcommand\RD{\sfR_{\rm D}}

\newcommand{\supp}{\mathop{\mathrm{supp}}\nolimits}

\newcommand{\dist}{\mathop{\mathrm{dist}}\nolimits}

\newcommand{\eps}{\varepsilon}

\renewcommand{\aa}{\alpha}

\theoremstyle{definition}

\usepackage[normalem]{ulem}

\makeatletter
\newcommand{\vast}{\bBigg@{3}}
\newcommand{\Vast}{\bBigg@{5}}

\usepackage{scrextend}
\changefontsizes{11.9pt}

\newcommand{\eg}{{\it e.g.}\,}
\newcommand{\ie}{{\it i.e.}\,}
\newcommand{\cf}{{\it cf.}\,}

\renewcommand\and{\qquad\text{and}\qquad}

\newcommand\sm{\setminus}

\newcommand{\comm}[1]{}

\renewcommand\aa{\alpha}
\newcommand\lm{\lambda}
\newcommand\s{\sigma}
\newcommand\p{\partial}

\newcommand\Omg{\Omega}

\newcommand\ii{{\mathsf{i}}}

\newcommand\one{\mathbbm{1}}

\newcommand{\spn}{\mathrm{span}\,}

\renewcommand\Re{{\rm Re}\,}
\renewcommand\Im{{\rm Im}\,}

\newcommand\arr{\rightarrow}

\newcommand\diag{{\rm diag}\,}

\newcommand\Sg{\Sigma}

\newcommand\sd{\sigma_{\rm d}}

\newcommand\dd{{\mathsf{d}}}

\newcounter{counter_a}
\newenvironment{myenum}{\begin{list}{{\rm(\roman{counter_a})}}%
{\usecounter{counter_a}
\setlength{\itemsep}{1.ex}\setlength{\topsep}{0.8ex}
\setlength{\leftmargin}{5ex}\setlength{\labelwidth}{5ex}}}{\end{list}}

\usepackage[latin1]{inputenc}
\usepackage[T1]{fontenc}

\numberwithin{figure}{section}
\numberwithin{equation}{section}
\theoremstyle{plain}
\newtheorem*{thm*}{Theorem}
\newtheorem{thm}{Theorem}[section]
\newtheorem{hyp}[thm]{Hypothesis}

\newtheorem{lem}[thm]{Lemma}
\newtheorem{prop}[thm]{Proposition}

\newtheorem{question}[thm]{Question}

\theoremstyle{remark}
\newtheorem{remark}[thm]{Remark}
\theoremstyle{plain}

\newcommand\ov{\overline}

\def\ov{\overline}
\newcommand\ran{{\rm ran\,}}

\def\sS{{\mathfrak S}}

      \def\dC{{\mathbb C}}
      
   \def\dH{{\mathbb H}}   
      
   \def\dN{{\mathbb N}}   
      \def\dR{{\mathbb R}}

\def\sfA{{\mathsf A}}      
      
   \def\sfH{{\mathsf H}}   \def\sfI{{\mathsf I}}

      \def\sfR{{\mathsf R}}
   \def\sfT{{\mathsf T}}

   \def\cH{{\mathcal H}}

\def\cS{{\mathcal S}}

\def\sfm{{\mathsf m}}

\def\N{\mathbb{N}}

\newcommand{\dom}{\mathrm{dom}\,}

\makeatletter
\def\section{\@startsection{section}{1}\z@{.9\linespacing\@plus\linespacing}%
	{.7\linespacing} {\fontsize{13}{14}\selectfont\bfseries\centering}}

\def\subsection{\@startsection{subsection}{1}\z@{.9\linespacing\@plus\linespacing}%
	{.3\linespacing} {\fontsize{11}{12}\selectfont\bfseries}}	
	
\def\paragraph{\@startsection{paragraph}{4}%
	\z@{0.3em}{-.5em}%
	{$\bullet$ \ \normalfont\itshape}}

\@addtoreset{equation}{section}
\makeatother

\makeatother

\begin{document}

\title[Diffusion operator with random jumps from the boundary]{\textsc{Spectral analysis of the multi-dimensional diffusion operator with random jumps from the boundary}}

\author{David Krej\v{c}i\v{r}\'{i}k}
\address{D.~Krej\v{c}i\v{r}\'{i}k, Department of Mathematics, Faculty of Nuclear Sciences and Physical Engineering, Czech Technical University in Prague, Trojanova 13, 120 00, Prague, Czech Republic}
\email{david.krejcirik@fjfi.cvut.cz}

\author{Vladimir Lotoreichik}
\address{V.~Lotoreichik, Department of Theoretical Physics, Nuclear Physics Institute, 	Czech Academy of Sciences, 25068 \v Re\v z, Czech Republic}
\email{lotoreichik@ujf.cas.cz}

\author{Konstantin Pankrashkin}
\address{K.~Pankrashkin,
Carl von Ossietzky Universit\"at, Institut f\"ur Mathematik, 26111 Oldenburg, Germany}
\email{konstantin.pankrashkin@uol.de}

\author{Mat\v{e}j Tu\v{s}ek}
\address{M.Tu\v{s}ek, Department of Mathematics, Faculty of Nuclear Sciences and Physical Engineering, Czech Technical University in Prague, Trojanova 13, 120 00, 
Prague, Czech Republic}
\email{matej.tusek@fjfi.cvut.cz}

\begin{abstract}
We develop a Hilbert-space approach to the diffusion process
of the Brownian motion in a bounded domain with random jumps from the boundary
introduced by Ben-Ari and Pinsky in 2007.
The generator of the process is introduced 
by a diffusion 
elliptic differential operator in the space of square-integrable functions, 
subject to non-self-adjoint and non-local boundary conditions 
expressed through a probability measure on the domain.
We obtain an expression for the difference
between the resolvent of the operator and that of its Dirichlet realization.
We prove that the numerical range 
is the whole complex plane, 
despite the fact that the spectrum is purely discrete and is contained in a half-plane.
Furthermore, for the class of absolutely continuous probability measures
with square-integrable densities
we characterise the adjoint operator and prove that the system of root vectors is complete.
Finally, under certain assumptions on the densities, 
we obtain enclosures for the non-real spectrum and find a sufficient condition 
for the non-zero eigenvalue with the smallest real part to be real.
The latter supports the conjecture of Ben-Ari and Pinsky that this eigenvalue is always real.
\end{abstract}

\maketitle


\section{Introduction}

Consider a Brownian motion in a bounded domain~$\Omega$
and wait until it hits a point of the boundary~$\partial\Omega$.
At the hitting time the Brownian particle gets restarted
at a point inside the domain according to 
a given probabilistic Radon measure~$\mu$ on~$\Omega$ 
and starts the diffusion afresh.
This process was introduced by Ben-Ari and Pinsky in~\cite{BP07,BP09}
and recently studied by Arendt, Kunkel and Kunze in~\cite{AKK16} 
(see also~\cite{K18} and references therein).

The underlying generator is given by an elliptic differential expression
on~$\Omega$ with a peculiar non-local boundary condition: it connects boundary values of functions with their values inside the domain. 
This generator can be realised 
as a closed linear operator in the Banach space $L^\infty(\Omega)$; see~\cite{AKK16, BP07}.

Our main aim is to rigorously define
the generator as a closed operator with compact resolvent
in the Hilbert space $L^2(\Omega)$. 
Furthermore, we compute its numerical range and analyse its basic spectral properties. 
In view of the conjecture from~\cite[Question~1]{BP07} 
about the reality
of the non-zero eigenvalue with the smallest real part for this operator, we are particularly interested in finding sufficient conditions for the low lying eigenvalues to be real.

\subsection{Setting and state of the art}	
Let $\Omega\subset\dR^d$, $d \ge 2$, be a bounded, connected, $C^2$-smooth 
open set
and let~$\mu$ be a probability measure on~$\Omega$.
Let $a\colon\Omega\to M_d(\dR)$ be a $d\times d$ symmetric matrix function with real-valued entries in $C^1(\overline{\Omega})$, such that
for some $c>0$ there holds
\begin{equation}
\label{axi}
\xi \cdot a(x)\xi \ge c \, |\xi|^2 
\quad\text{ for all $\xi\in \dR^d$ and all $x\in\Omega$,}
\end{equation}
where the dot $\cdot$ denotes the scalar product in~$\dR^d$.
The Brownian motion of interest~\cite{AKK16, BP07, BP09, K18} is generated by the following linear operator in the Banach space $L^\infty(\Omega)$:  
\begin{equation}\label{eq:operator}
	\begin{aligned}
		\sfA_\mu u & := -\nabla \cdot  a \,\nabla u,\\
		\dom\sfA_\mu &:= 
		\left\{u\in C(\ov\Omega)\cap W(\Omega)\colon \nabla \cdot a\,\nabla u \in L^\infty(\Omega),\, u|_{\p\Omega}= \int_\Omega u\dd\mu\right\},	
	\end{aligned} 
\end{equation}
where the auxiliary space $W(\Omega)$ is defined as
\[
	W(\Omega) := \bigcap_{p > 1} W^{2,p}_{\rm loc}(\Omega);  
\]
here $W^{2,p}_{\rm loc}(\Omg)$ is the local $L^p$-based second-order Sobolev space on $\Omg$. It is proved in~\cite[Thm.~1.6]{AKK16} (see also~\cite{BP09} for a probabilistic description) that $\sfA_\mu$ is a closed operator and that $-\sfA_\mu$ generates
a holomorphic contraction positive semigroup $\sfT_\mu(t)$, $t > 0$, in the Banach space $L^\infty(\Omega)$. According to~\cite[Thm.~4.8 and Cor.~5.8]{AKK16}, the operator $\sfA_\mu$ has only point spectrum, which lies in the half-plane $\{\lm\in\dC\colon \Re\lm \ge 0\}$. Moreover, the property $\s(\sfA_\mu)\cap \ii\dR = \{0\}$ holds.

Let us introduce the notation 
\[
	\lm_1(\mu) := \inf\big\{\Re\lm\colon\lm\in\s(\sfA_\mu)\sm\{0\}\big\} \ge 0 \,,
\]
with the convention that $\lm_1(\mu)  =\infty$ if $\s(\sfA_\mu) = \{0\}$
(notice that, in the most general case, 
it has not been proved that $\sfA_\mu$ possesses non-zero eigenvalues). 
According to~\cite[Thm.~1.3 (d) and Cor.~5.10\,(2)]{AKK16}, there is a non-negative function
\[
h \in L^1(\Omega), \quad \int_\Omega h(x) \dd x = 1,
\]
such that for any $\eps < \lm_1(\mu)$ there is a constant $M > 0$ 
such that
\begin{equation}\label{eq:largetime}
	\left\| \sfT_\mu(t) f - 
	\one\int_\Omega f(x)h(x)\dd x \right\|_\infty \le Me^{-\eps t}
	\text{ for all $t>0$ and all $f\in L^\infty(\Omega)$},
\end{equation}
where $\one$ is the constant function $x\mapsto 1$.
Thus, estimates of $\lm_1(\mu)$ are of 
clear probabilistic interest~\cite{BP07}. A refinement of the large time behavior~\eqref{eq:largetime} is obtained in~\cite[Thm. 1]{BP09} under extra assumptions. In particular, according to~\cite{BP09}, the function $h$ can be explicitly expressed through
the measure~$\mu$ and the Green's function corresponding to the Dirichlet realization of the operator. 

The following striking question posed by Ben-Ari and Pinsky in~\cite{BP07} remains open.
\begin{question}
	Is $\lm_1(\mu)$ an eigenvalue of $\sfA_\mu$?
\end{question}	
\noindent
In the case $0<\lm_1(\mu)<\infty$, the question means whether 
the non-zero eigenvalue with the smallest real part is real.

Sufficient conditions for the entire spectrum of $\sfA_\mu$ to be real are obtained in~\cite{BP07}. In particular, the spectrum of $\sfA_\mu$ is real for $\mu$ being the uniform probability measure on $\Omega$ and for 
\[
\dd\mu(x) = \frac{\chi_1(x)}{(\chi_1,\one)_{L^2(\Omega)}}\dd x,
\]
where
 $\chi_1> 0$ is the $L^2$-normalized ground-state of $-\nabla\cdot a\nabla$ with Dirichlet boundary condition.
An explicit example with non-real eigenvalues is constructed in~\cite[Rem.~1.4]{LLR}, however, in this example $\lm_1(\mu)$ is still a (real) eigenvalue.
The one-dimensional version of the problem is considered in the series of papers
\cite{LLR, KW11, B14, KK16,  Y18, SY19}. 
Surprisingly, the spectrum turns out to be always real~\cite[Thm.~1.2]{LLR} in that setting.
In~\cite{KK16} the reality of the spectrum was explained 
through a generalised similarity of the non-self-adjoint generator
to a self-adjoint operator in the Hilbert space $L^2(\Omega)$.

\subsection{Main results}
Assuming that $\mu$ is a Radon probability measure on $\Omega$ defining 
a continuous functional on the Sobolev space $H^2(\Omega)$, we introduce a linear operator 
\[
	\Op u := -\nabla \cdot a\,\nabla u,\qquad \dom\Op := \left\{u\in H^2(\Omega)\colon u|_{\p\Omg}=\int_\Omega u \dd \mu \right\},
\]
in the Hilbert space $L^2(\Omega)$.
The operator~$\Op$
can be viewed as an extension of $\sfA_\mu$ to a larger space $L^2(\Omega)\supset L^\infty(\Omega)$.  We check that $\Op$ is densely defined, closed, and non-self-adjoint and that the spectral data of $\sfH_\mu$  are
the same as those of $\sfA_\mu$; \ie, these operators have the same spectra and the same families of eigenfunctions.

Our goal is to perform a spectral analysis of~$\Op$
in the Hilbert-space setting.
In Theorem~\ref{thm0} we derive a Krein-type resolvent formula for the resolvent difference of $\Op$ and its Dirichlet counterpart on $\Omega$ 
and obtain a Birman--Schwinger-type characterisation of the spectrum for $\Op$.
As an easy consequence of that, we show that the spectrum of $\Op$ is purely discrete and invariant under complex conjugation. Furthermore, 
in Theorem~\ref{thm:numerical_range}
we establish that the numerical
range of $\Op$ is the whole complex plane.

Particular attention is then paid to the measures $\dd  \mu(x) = w(x)\dd x$ with $L^2$-densities $w$. For such measures we are able to compute and characterise the adjoint operator $\Op^*$  and to  prove that the system of root vectors of $\Op$ is complete in the Hilbert space $L^2(\Omg)$. It means, in particular, that the spectrum of $\Op$ consists of infinitely many isolated
points that accumulate at complex infinity. For $w$ being either a perturbation of the constant function $\frac{\one}{|\Omega|}$ or of the $L^1$-normalized ground-state of the Dirichlet realization of $(-\nabla\cdot a\,\nabla)$ we obtain enclosures on the non-real eigenvalues; \cf Theorems~\ref{thm1} and~\ref{thm3}.  

Next, we discuss these enclosures in more detail.
Let $\{\lm_k\}_{k=1}^\infty$ be the Dirichlet eigenvalues of $(-\nabla\cdot a\,\nabla)$ on $\Omega$
enumerated in the non-decreasing order and counted with multiplicities. Let $\{\chi_k\}_{k=1}^\infty$ be the respective real-valued eigenfunctions
normalised to~$1$ in $L^2(\Omega)$, chosen so that  $\chi_1 > 0$.
We consider two special classes of measures.

\subsubsection*{Class I}
For the measures of the structure 
\begin{equation} \label{eq:measure1}
	\boxed{\dd\mu(x)
= \left(\frac{\one(x)}{|\Omega|} + v(x) \right)\dd x}
\end{equation}
with real-valued $v\in L^2(\Omega)$ satisfying, for some $k\in\dN$,
\begin{equation}\label{eq:v}
	 v \ge -\frac{\one}{|\Omega|},\quad
	 \int_\Omega v(x)\dd x= 0,\quad
	 \int_\Omega v^2(x)\dd x < \
	 \frac{1}{|\Omega|^2}
	 \min_{1\le n\le k}\left\{\left(\int_\Omega \chi_n(x)\dd x\right)^2\right\},
\end{equation}
we obtain in Theorem~\ref{thm1} that there are no eigenvalues of $\sfH_\mu$ in the set
\begin{equation*}
	\mathbb{H}_k 
	= \left\{\lm\in\dC\setminus\dR\colon \Re \lm \le \frac{\lm_k+ \lm_{k+1}}{2}\right\}.
\end{equation*}
As a result, all eigenvalues $\lambda$ of $\Op$ with $\Re \lambda \le \frac{\lm_k+ \lm_{k+1}}{2}$
are real.

The analysis of these eigenvalues 
in Theorem~\ref{thm2} yields, under the additional assumption of
simplicity of the eigenvalues $\lm_1,\lm_2,\dots,\lm_k$ 
in the spectrum of $\OpD$,
that in each interval $(\lm_i,\lm_{i+1})$, $i=1,2,\dots, k-1$ there is exactly one
eigenvalue of $\Op$. In particular,	
under the assumption~\eqref{eq:v} on $v$ with $k= 2$ and in the case that $\lm_2$ is a simple
eigenvalue of $\OpD$
we conclude in Proposition~\ref{prop:real} with an extra argument that $\lm_1(\mu)$ 
is a real eigenvalue
and obtain an estimate on it.
These results are interesting to compare with~\cite[Thm. 4\,(i)]{BP07}, where it is proved that all the eigenvalues of $\Op$ are real, 
provided that the numerical sequence 
\[
  \xi_n := (\one,\chi_n)_{L^2(\Omg)}\int_\Omg \chi_n(x) \dd\mu(x),\qquad n\ge 2,
\] 
satisfies either $\xi_n \ge 0$ or $\xi_n \le 0$ for all $n\ge 2$. The conclusion that we get is weaker, but on the other hand our
assumption on the measure involves only eigenfunctions corresponding to low eigenvalues. 

\subsubsection*{Class II}
For the measures of the structure 
\begin{equation} \label{eq:measure2}
\boxed{\dd\mu(x)
= \left(\frac{\chi_1(x)}{(\chi_1,\one)_{L^2(\Omega)}} + v(x)\right)\dd x}
\end{equation} 
with real-valued $v\in L^2(\Omega)$ satisfying 
\begin{equation}\label{eq:v2}
v \ge -\frac{\chi_1}{(\chi_1,\one)_{L^2(\Omega)}},\qquad
\int_\Omega v(x)\dd x= 0,\qquad
\int_\Omega v^2(x)\dd x < \frac{1}{|\Omega|},
\end{equation}
we obtain in Theorem~\ref{thm3} the enclosure
\begin{equation}\label{eq:encl_intro}
	\s(\Op)\sm [0,\infty) \subset\left\{\lm\in\dC\sm\dR\colon\frac{\dist(\lm,\s(\OpD))}{|\lm_1-\lm|} \le |\Omega|^{1/4}\|v\|_{L^2(\Omega)}^{1/2}\right\}, 
\end{equation}
\ie, the non-real spectrum of $\Op$ belongs to a neighbourhood of $\s(\OpD)$, the size of which is controlled by the norm of $v$ in $L^2(\Omega)$, see Figure~\ref{fig}.
\begin{figure}[h] 
 \centering
 \includegraphics[width=9cm]{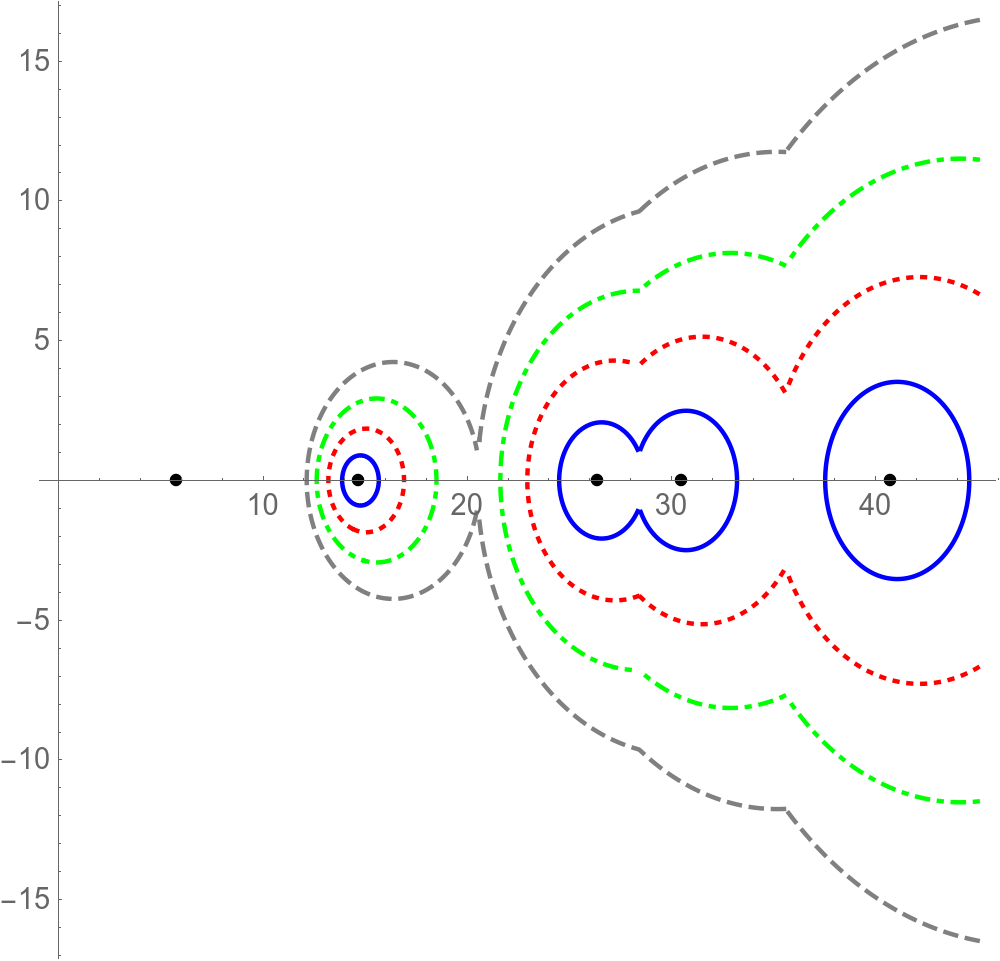} 
 \caption{Matryoshka-type spectral enclosure~\eqref{eq:encl_intro} obtained in Theorem~\ref{thm3}. We chose the unit disk for $\Omega$,
$a(x):=\diag(1,\dots,1)$
 and put $|\Omg|^{1/4}\|v\|^{1/2}_{L^2(\Omg)}=0.1$ (the blue solid line), $0.2$ (the red dotted line), $0.3$ (the green dashed-dotted line), $0.4$ (the gray dashed line). The black points are the Dirichlet eigenvalues.}
 \label{fig}
\end{figure}
\subsection{Structure of the paper}
The paper is organized as follows. In Section~\ref{sec:def} we define the operator $\Op$, derive a Krein-type resolvent formula and analyse basic properties of $\Op$. The numerical range of $\Op$ is studied in Section~\ref{sec:numrange}. 
The case of $\dd\mu(x) = w(x)\dd x$ with a square-integrable
density~$w$ is considered in Section~\ref{sec:L2}, in which also the adjoint of $\Op$ is investigated.
The enclosures for the non-real spectrum of $\Op$ and estimates for the low lying real eigenvalues are obtained in Section~\ref{sec:enclosures}. 
Namely, the measures of the structures~\eqref{eq:measure1} and~\eqref{eq:measure2} are analysed in Subsections~\ref{ssec:uniform} and~\ref{ssec:chi1}, respectively. 

\section{Definition of the operator and Krein-type resolvent formula}
\label{sec:def}
Let $\Omg\subset\dR^d$ be a bounded 
$C^{2}$-domain with the boundary $\p\Omg$
and $\mu$ be a Radon probability measure on $\Omg$ defining a continuous functional
\[
	H^2(\Omg)\ni u \mapsto 
	\langle u\rangle_\mu
	:= \int_\Omg u \,\dd \mu. 
\] 
\begin{remark}
An important subclass of such measures
is given by $\dd\mu(x) = w(x)\dd x$ with
a non-negative real-valued function $w \in L^2(\Omg)$ with $\|w\|_{L^1(\Omega)}=1$. Another interesting subclass is of the type
\[
	\mu(U) := \frac{1}{|\Sg|} \int_{U\cap \Sg} \dd \s, \qquad U\subset\Omg,
\]
where $\Sg\subset\Omg$ is a $C^\infty$-smooth
closed hypersurface and $\dd\s$ is the $(d-1)$-dimensional
Hausdorff measure on it. Let us also point out that  in dimensions $d = 2, 3$ we can treat the Dirac measure supported on a point in $\Omega$.
\end{remark}

We consider a linear operator $\Op$ defined as follows:
\begin{equation}\label{eq:Op}
	\Op u := -\nabla \cdot a \,\nabla u,\qquad
	\dom\Op := 
	\left\{ u \in H^2(\Omg)\colon
	\langle u \rangle_\mu = u|_{\p\Omg}\right\}.
\end{equation}
\begin{remark}
Note that at the beginning we could have assumed that $\mu$ is supported on $\overline\Omega$ but not exclusively on $\partial\Omega$. 
However, this apparently more general setting is actually
already contained in the definition above, without loss of generality. 
Indeed,
let $\mu=\mu_1+\mu_2$, where $\mu_1$ and $\mu_2$ are supported on $\Omega$ and $\partial\Omega$, respectively, and $\mu_2(\partial\Omega)<1$. Then the boundary condition for $\dom\Op$, where we integrate over $\overline\Omg$, reads 
$$u|_{\p\Omg}=\int_\Omega u\dd\mu_1+\mu_2(\p\Omg)u|_{\p\Omg},$$
i.e. $\langle u \rangle_{\tilde\mu} = u|_{\p\Omg}$ with the new probability measure
$\tilde\mu=\mu_1/(1-\mu_2(\p\Omg))$ supported on $\Omega$.
\end{remark}
\begin{remark}
	We also remark that in~\cite{AKK16} a more general boundary condition of the form
	\[
		u(x) = \int_\Omg u(y) \, \dd \mu_x(y), \quad x\in\p\Omg,
	\]
	for a position-dependent measure $\mu_x$ is considered. It reduces to the boundary condition in~\eqref{eq:Op} provided that $\mu_x = \mu$ for all $x\in\p\Omg$.
\end{remark}
\begin{prop} \label{prop:density}
$\Op$ is densely defined in $L^2(\Omega)$.
\end{prop}
\begin{proof}
Let $\varphi\in L^2(\Omega)$ and $\varepsilon>0$, then one can find $\psi\in C_0^\infty(\Omega)$
which satisfies  $\|\varphi-\psi\|_{L^2(\Omega)}<\frac{\varepsilon}{2}$.

Let us define the double-sided open neighbourhood of $\p\Omg$ as 
	\[
		\Omega_\eps:=\{x\in\dR^d:\, \dist(x,\partial\Omega)<\eps\}.
	\]
Furthermore, let 
$\one_\eps$ 
be the indicator function of $\Omega_\eps$, and let $(f_\eps)_{\eps>0}$ be the standard family of mollifiers in $\dR^d$; \cf \cite[Sec. 7.2]{GT}. Then the convolution
	\[
		g_n:=\one_{\frac{2}{n}}*f_{\frac{1}{n}},\qquad n\in\dN, 
	\]
	is a $C_0^\infty(\dR^d)$-function with values in $[0,1]$ that equals $1$ on $\Omega_{\frac{1}{n}}$ and is zero away from $\Omega_{\frac{3}{n}}$.
	Hence, $|g_n|\le 1\in L^1(\Omega,\dd\mu)$, and for any $x\in\Omega$ one has $\lim_{n\to \infty}g_n(x)=0$.
	Then the dominated convergence implies
	\begin{equation} \label{eq:gn}
	\lim_{n\to\infty}\int_\Omega g_n\dd\mu=0.
	\end{equation}
	
	Put
	\[
	\psi_n:=\psi+C_n g_n\vert_\Omega\quad\text{with}\quad C_n:=\dfrac{\int_\Omega\psi\dd\mu}{1-\int_\Omega g_n\dd\mu}.
	\]
	Clearly, $\psi_n\in H^2(\Omega)$ and $\langle\psi_n\rangle_\mu=C_n=\psi_n\vert_{\partial\Omega}$, 
	\emph{i.e.}, $\psi_n\in\dom\Op$. For every sufficiently large $n$ one has
	\[
	\Big|\int_\Omega g_n\dd\mu\Big|\le\dfrac{1}{2}, \quad |C_n|\le 2 \Big|\int_\Omega\psi\dd\mu\Big|,
	\]
    due to \eqref{eq:gn}. Hence, we get
	\[
		\|\psi-\psi_n\|_{L^2(\Omega)}= |C_n|\|g_n\|_{L^2(\Omega)}\leq 2 \Big|\int_\Omega\psi\dd\mu\Big| \|\one_{\frac{3}{n}}\|_{L^2(\Omega)}
		\le 2 \Big|\int_\Omega\psi\dd\mu\Big| \, \sqrt{|\Omega_{\frac{3}{n}}|},
	\]
	where $\one_{\frac{3}{n}}$ stands for the indicator function of $\Omega_{\frac{3}{n}}$.
	 Since $\lim_{n\to\infty}|\Omega_{\frac{3}{n}}|=0$, one can choose $N$ to have $\|\psi-\psi_N\|_{L^2(\Omega)}<\frac{\varepsilon}{2}$, and then
	\[
	\|\varphi-\psi_N\|_{L^2(\Omega)}\le 	\|\varphi-\psi\|_{L^2(\Omega)}+	\|\psi-\psi_N\|_{L^2(\Omega)}< \dfrac{\varepsilon}{2}+\dfrac{\varepsilon}{2}=\varepsilon.
	\]
	As $\varphi\in L^2(\Omega)$ and $\varepsilon>0$ were arbitrary and $\psi_N\in \dom\Op$, the claim follows.
\end{proof}

Now, introduce the Dirichlet realization $\OpD$ of $(-\nabla\cdot a\,\nabla u)$
on $\Omg$,
\begin{equation}\label{eq:OpD}
	\OpD u := -\nabla\cdot a\,\nabla u,\qquad
	\dom\OpD := H^2(\Omg)\cap H^1_0(\Omg),
\end{equation}	
which is self-adjoint in the Hilbert space $L^2(\Omg)$ (it is obvious
that the operator is symmetric, then the self-adjointness follows from its surjectivity, see Lemma~\ref{lem:regularity} below). 
With this definition,
the domain of $\Op$ can be alternatively described as
\begin{equation}\label{eq:Op_dom}
	\dom\Op = 
	\left\{ u = u_0 +c \colon
	u_0 \in \dom\OpD,
	\langle u_0 \rangle_\mu = 0,
	c \in \dC\right\}.
\end{equation}
For the sake of convenience, we employ the 
abbreviation
\[
	\sfR_{\rm D}(\lm) := (\OpD - \lm)^{-1},
	\qquad
	\lm\in\rho(\OpD),
\]
and introduce the following complex-valued function
\[
	\dC\sm\s(\OpD)\ni\lm\mapsto \sfm_\mu(\lm)
	:= \langle\RD(\lm)\one\rangle_\mu.
\]
Recall that the resolvent~$\sfR_{\rm D}(\lm)$ is compact. Let
$\{\lambda_n\}_{n=1}^{\infty}$ stand for the eigenvalues of $\OpD$ counted with multiplicities in non-decreasing order 
and let $\{\chi_n\}_{n=1}^{\infty}$ be the corresponding real-valued orthonormal eigenfunctions of $\OpD$. 
The ground state~$\chi_1$ can be chosen to be strictly positive in $\Omega$.
Then we can write
\begin{equation} \label{eq:mSum}
\sfm_\mu(\lm)=\Big\langle
\RD(\lm)\sum_{n=1}^{+\infty}(\one,\chi_n)_{L^2(\Omg)}\chi_n
\Big\rangle_{\!\mu}=\sum_{n=1}^{+\infty}\frac{(\one,\chi_n)_{L^2(\Omg)} \langle\chi_n\rangle_\mu}{\lambda_n-\lambda}.
\end{equation}
Note that the sum in the first equality converges in the $L^2$-norm: recall that $\langle\cdot\rangle_\mu$ 
is assumed to be continuous on $H^2(\Omega)$, 
and then $\langle\RD(\lm)\cdot\rangle_{\mu}$ is continuous on $L^2(\Omega)$.
This justifies the second equality in \eqref{eq:mSum}. Moreover, we see that $\sfm_\mu$ is meromorphic on $\dC$ and not everywhere zero. 
Indeed, assuming on the contrary that $\sfm_\mu\equiv0$ identically, 
then in particular
\[
	(\one,\chi_1)_{L^2(\Omg)} \langle\chi_1\rangle_\mu= \lim_{\lm\arr\lm_1}(\lm_1-\lm)\sfm_\mu(\lm) = 0,
\]	
which contradicts the positivity of~$\chi_1$.
Similarly, employing the identity theorem from complex analysis, one can infer that the set of zeros of $\sfm_\mu$ has no finite accumulation point and, therefore, it is at most countable.

Now we are ready to formulate the theorem,
which contains basic properties of $\Op$
and a Krein-type resolvent identity.
\begin{thm}\label{thm0}
	Let the operators $\Op$
	and $\OpD$ be as in~\eqref{eq:Op}
	and in~\eqref{eq:OpD}, respectively. 
	Define the set  
	$\cS_\mu := 
	\{\lm\in\dC\sm\s(\sfH_{\rm D})\colon \lm\sfm_\mu(\lm) \ne 0\}$.	Then the following
	hold.
	\begin{myenum} 
	\item 
	For any $\lm \in
	\cS_\mu $ the operator
	\begin{equation} \label{eq:resolvent}
	\sfR_\mu(\lm) 
	:= 
	(\Op - \lm)^{-1}
	=
	\sfR_{\rm D}(\lm) 	-
	\left(\lm\sfR_{\rm D}(\lm)\one  + \one\right)
	\frac{\langle\sfR_{\rm D}(\lm)\,\cdot\rangle_\mu}{\lm\sfm_\mu(\lm)} 
	\end{equation}
	is well defined and compact in $L^2(\Omg)$.
	In particular,
	$\cS_\mu\subset\rho(\Op)$, the linear operator $\Op$ is closed, and its
	spectrum  is purely discrete.
	\item $\ker\Op = \spn\{\one\}$.
	\item	
	For any $\lm\in\rho(\OpD)\setminus\{0\}$,
	$\sfm_\mu(\lm) = 0$ if, and only if,
	$\lm\in\sd(\Op)$. In the positive case,
	the geometric multiplicity of $\lm$ is one and 
    $\lm\RD(\lm)\one +\one$ is (up to a multiplicative constant) the eigenfunction
	of $\Op$ corresponding to the eigenvalue~$\lm$. 
\end{myenum}
\end{thm}	

\begin{proof}
    (i) Firstly, note that $\cS_\mu\neq\emptyset$. 
	Let $\lm\in\cS_\mu$ be fixed and let $v\in L^2(\Omg)$ be arbitrary.
	Consider the resolvent equation
	$(\Op - \lm)u = v$. Using the representation~\eqref{eq:Op_dom} of $\dom\Op$,
	we find for $u = u_0 + c\in\dom\Op$
	with $u_0 \in \dom\OpD$ and $c\in\dC$ that
	$\OpD u_0  - \lm u_0 = v + \lm c$.
	Hence, we get
	\begin{equation}\label{eq:u0}
		u_0 = \RD(\lm) v + \lm c\RD(\lm)\one.
	\end{equation}
	Furthermore, 
	employing that $\langle u_0\rangle_\mu = 0$,
	we obtain
	\[
		\langle\RD(\lm) v\rangle_\mu + \lm c\sfm_\mu(\lm) = 0.
	\] 
	Hence, the constant $c\in\dC$ can be expressed
	as follows
	\[
		c = 
		-\frac{ \langle\RD(\lm) v\rangle_\mu }
		{\lm\sfm_\mu(\lm)}.
	\]
	Substituting the above expression for $c$
	into~\eqref{eq:u0} we obtain
	the identity
	\[
		u_0 = \RD(\lm) v
		- \lm\RD(\lm)\one
		\frac{ \langle\RD(\lm) v\rangle_\mu }
		{\lm\sfm_\mu(\lm)}.
	\]
	Finally, we get
	\[
		u = u_0 + c
		=
		\RD(\lm) v
		- \lm\RD(\lm)\one
		\frac{ \langle\RD(\lm) v\rangle_\mu }
		{\lm\sfm_\mu(\lm)}
		-\frac{ \langle\RD(\lm) v\rangle_\mu }
		{\lm\sfm_\mu(\lm)}.
	\]
	Hence, we conclude that the operator
	$\Op - \lm$ is invertible and the inverse
	\[
		(\Op - \lm)^{-1} := 	\RD(\lm) 
		- \big(\lm\RD(\lm)\one + \one\big)
		\frac{ \langle\RD(\lm) \cdot\rangle_\mu }
		{\lm\sfm_\mu(\lm)}
	\] is everywhere defined
	in $L^2(\Omg)$.  
	
	Since the operator $(\Op -\lm)^{-1}$ is bounded and everywhere defined in $L^2(\Omg)$, we conclude that $\lm\in\rho(\Op)$. Therefore, the inclusion
	$\cS_\mu\subset\rho(\Op)$ holds and  $\Op$ is closed. 
	
	Recall that $\langle\RD(\lm)\cdot\rangle_\mu$ is a continuous functional on $L^2(\Omega)$.
	Therefore, the resolvent of $\Op$
	is a rank-one perturbation of the resolvent of $\OpD$. Hence, compactness of the resolvent of $\Op$ is a consequence of the same property for $\RD(\lm)$. Thus,
	discreteness of the spectrum for $\Op$
	follows.
	
	(ii) 
	First, it is easy to see that $\one \in\dom\Op$
	and that $\sfH_\mu \one = 0$. 
	Second, suppose now that $u\in\dom\Op$ and 
	$\Op u = 0$. By~\eqref{eq:Op_dom}, we can decompose $u = u_0 + c$ with $u_0 \in\dom\OpD$
	and $c\in\dC$. Hence, we end up with 
	$\OpD u_0 = 0$ and therefore $u_0 = 0$.
	Combining the above two observations,
	we get $\ker\Op = \spn\{\one\}$.
	
	(iii)
	Let $\lm\in\rho(\OpD)$ be such that
	$\sfm_\mu(\lm) = 0$. Set 
	$u := \lm\sfR_{\rm D}(\lm) \one +\one$.
	Clearly, $\langle u  \rangle_\mu = u|_{\p\Omg} = 1$
	and thus $u\in\dom\Op\setminus\{0\}$. Moreover,
	\[
	\begin{aligned}
		\sfH_\mu u
		& =
		\sfH_\mu 
		\left(\lm\sfR_{\rm D}(\lm) \one + \one\right)
		= \OpD 	
			\lm\sfR_{\rm D}(\lm) \one\\
			&=(\OpD-\lambda)\lm\sfR_{\rm D}(\lm) \one
		+
		\lm^2
		\sfR_{\rm D}(\lm) \one\\
		& =
		\lm\one + \lm^2\sfR_{\rm D}(\lm)\one  
		= \lm u.
	\end{aligned}
	\]
	Now, suppose that $\lm\in\sd(\Op)\cap\rho(\OpD)\setminus\{0\}$.
	Hence, there exists $u\in\dom\Op\setminus\{0\}$
	such that
	$\Op u = \lm u$. 
	By~\eqref{eq:Op_dom} we can decompose $u$ as follows
	%
	$u = u_0 + c.$
	%
	Since $\lm\notin\s(\OpD)$ we conclude that
	$c\ne 0$.
	Using that $\Op u = \lm u$ we get
	%
	$	\OpD u_0 = \lm u_0 + \lm c.$
	%
	Hence,
	%
$     
		u_0 = \lm c\RD(\lm)\one,
$
	%
	and, finally,
	%
$
		\lm c \sfm_\mu(\lm) = 0. 
$
\qedhere
\end{proof}

In the next two propositions we collect basic spectral properties of $\Op$.

\begin{prop}\label{prop:conj_ev}
	If $\lm\in\dC\sm\dR$ is an eigenvalue of $\Op$ with an eigenfunction $u$,
	then $\ov\lm$ is also an eigenvalue of $\Op$ with an eigenfunction $\ov u$. 
	Moreover, the geometric multiplicities of $\lm$ and $\ov\lm$ coincide.
\end{prop}	
\begin{proof}
	Let $\lm\in\dC\sm\dR$ be an eigenvalue of $\Op$ with an eigenfunction $u\in \dom\Op\sm\{0\}$. Obviously, $\ov u \in\dom \Op$ and moreover,
	$\Op \ov u = \ov{\Op u} = \ov{\lm u} = \ov\lm \ov u.$
	Next, 
it is easy to see that the complex conjugation is a bijection between
	$\ker(\Op - \lm)$ and $\ker(\Op - \ov\lm)$. Hence, the geometric multiplicities of these eigenvalues coincide.
\qedhere
\end{proof}	

\begin{prop}\label{prop:0_lm1}
	$\s(\Op)\cap(-\infty,\lm_1) = \{0\}$, where $\lm_1$ is the lowest eigenvalue of $\OpD$.
\end{prop}	
\begin{proof}
	This statement is almost trivial. Indeed,
	for any $\lm\in (-\infty,\lm_1)$  one has the pointwise inequality
	$\RD(\lm)\one > 0$ due to the positivity improving of $\RD(\lm)$; \cf~\cite[Thm. 1.3.2 and Lem. 1.3.4]{D89}.
	Hence, $\sfm_\mu(\lm)>0$ and the claim follows from Theorem~\ref{thm0}\,(iii). 
\end{proof}

We use the resolvent formula to show explicitly the absence of the self-adjointness:
\begin{prop}\label{prop}
For any choice of $\mu$, the operator $\Op$ is not self-adjoint.
\end{prop}

\begin{proof}
Assume the opposite, \ie\ that $\Op$ is self-adjoint. Due to Proposition~\ref{prop:0_lm1} one has $-1\notin\s(\Op)$,
and then the operator $\sfR_\mu(-1)=(\Op+1)^{-1}$ is bounded and self-adjoint.
Due to the assumption on $\mu$, the map 
\[
	L^2(\Omg)\ni u\mapsto \big\langle\RD(-1)u\big\rangle_\mu\in\dC
\]is a bounded linear functional, and by the Riesz theorem one can find $g\in L^2(\Omg)$ such that
$\big\langle \RD(-1)u\big\rangle_\mu=( u,g)_{L^2(\Omg)}$ for all $u\in L^2(\Omg)$.
By Theorem~\ref{thm0} one has
\[
\sfR_\mu(-1) 
	=
	\RD(-1) 	+ \dfrac{1}{\sfm_\mu(-1)} (\cdot,g )_{L^2(\Omg)} f,
	\quad f:=	-\RD(-1)\one  + \one,
\]
and then
\[
\sfR_\mu(-1)^* 
	=
	\RD(-1) 	+ \dfrac{1}{\sfm_\mu(-1)} (\cdot,f)_{L^2(\Omg)}  g.
\]
Due to $\sfR_\mu(-1)=\sfR_\mu(-1)^*$ we have then $f=g$, i.e. $( u,g)_{L^2(\Omg)}=(u,f)_{L^2(\Omg)}$
for all $u\in L^2(\Omg)$, or, in greater detail,
\begin{align*}
\int_\Omg \RD(-1) u\, \dd \mu&=\big(u, -\RD(-1)\one  + \one\big)_{L^2(\Omg)}\\
&= \big(\OpD\RD(-1) u,\one\big)_{L^2(\Omega)}  \text{ for all } u\in L^2(\Omg).
\end{align*}
Denoting $\varphi:=\RD(-1) u$ we arrive at
\[
\int_\Omg \varphi\,\dd\mu= \int_\Omg 
(-\nabla\cdot a\nabla \varphi) \, \dd x\quad \text{ for all } \varphi\in\dom \OpD.
\]
In particular, for all $\varphi\in C^\infty_0(\Omega)$ one obtains, using the integration by parts on the right-hand side
of the last equality,
\[
\int_\Omg \varphi\,\dd\mu=\int_\Omg a\nabla \varphi\cdot \nabla \one\,\dd x=0,
\]
and it follows that $\mu$ is not a probability measure on $\Omg$.
\end{proof}


In the following proposition we show that the spectral data of $\sfA_\mu$ and $\Op$ coincide. As a result, some known spectral properties of $\sfA_\mu$ 
can be transferred to $\Op$.
\begin{prop}\label{prop:spectra}
	The eigenvalues and the respective eigenspaces of $\sfA_\mu$ and $\Op$ coincide.
	In particular, $\s(\Op)\subset\{\lm\in\dC\colon \Re\lm \ge 0\}$ and $\s(\Op)\cap\ii\dR = \{0\}$.
\end{prop}
The proof relies on the well-known regularity lemma: 
\begin{lem}\label{lem:regularity}\cite[Thm. 2.4.2.5]{G85}
	For any $p\ge 1$ and any $f\in L^p(\Omega)$ there is a unique $w\in W^{2,p}(\Omega)\cap W^{1,p}_0(\Omega)$ 
	with $\nabla \cdot a\,\nabla w=f$.
	%
\end{lem}
\begin{proof}[Proof of Proposition~\ref{prop:spectra}]
	\emph{Step 1.} 
	Let $\lm\in\dC$ and $u\in\dom\sfA_\mu$ be such that $\sfA_\mu u = \lm u$.
	Our aim is to show that $u\in\dom\Op$ and that $\Op u = \lm u$.
	The function $u$ can be decomposed as $u = u_0 + c$ with $c\in\dC$ and
	with $u_0\in C(\ov\Omega)$
	satisfying $\nabla\cdot a\nabla u_0\in L^\infty(\Omega)$ and $u_0|_{\p\Omega} =
	\int_\Omg u_0(x)\dd\mu(x) =  0$.
	The inclusions $C(\ov\Omega)\subset L^\infty(\Omega)\subset L^2(\Omega)$ yield
	$u_0, \nabla\cdot a\nabla u_0  \in L^2(\Omega)$. Hence, by Lemma~\ref{lem:regularity} with $p = 2$ we obtain that
	$u_0 \in H^2(\Omega)$ and hence also $u = u_0 + c\in H^2(\Omega)$. Therefore, $u\in\dom\Op$
	and
	$\Op u = \lm u$ easily follows.  
	
	\emph{Step 2.} 
	Let $\lm\in\dC$ and $u\in\dom\Op$ be such that $\Op u = \lm u$.
	Our aim is to show the converse implication, namely, that $u\in\dom\sfA_\mu$ and that $\sfA_\mu u = \lm u$.
	Clearly, the eigenfunction $u$ can be decomposed as $u = u_0 + c$ with $c\in\dC$ with 
	$u_0 \in H^2(\Omega)$ satisfying $u_0|_{\p\Omega} = 0$.
	Recall  the identity 
	\begin{equation}\label{eq:relation}
		\OpD u_0 = \lm u_0 + \lm c.
	\end{equation}
	If $d \le 4$, the Sobolev embedding theorem~\cite[Eq. 1.4.4.5]{G85} yields $u_0\in L^p(\Omega)$ for all $p\ge 1$.
	Hence, we conclude from~\eqref{eq:relation} and Lemma~\ref{lem:regularity} that $u_0 \in W^{2,p}(\Omega)$ for all $p \ge 1$.
	If $d > 4$ we use standard bootstrap argument based on Lemma~\ref{lem:regularity} applied several times in order to show that again $u_0 \in W^{2,p}(\Omega)$ for all $p \ge 1$. 
	Since by the Sobolev embedding~\cite[Eq. 1.4.4.6]{G85} we have $W^{2,p}(\Omega)\subset C^{0,2-\frac{d}{p}}(\ov\Omega)$ for 
	$1 < \frac{d}{p} < 2$. Hence, we get with $p := \frac{d^2}{d+1}$ that $u_0\in C^{0,1-\frac{1}{d}}(\ov\Omega)\subset C(\ov\Omega)$. Therefore, $u = u_0+c\in C(\ov\Omega)$ and thanks to the identity~\eqref{eq:relation} we have $\nabla \cdot a\,\nabla  u = \nabla \cdot a\,\nabla u_0 \in L^\infty(\Omega)$.
	Thus, $u\in\dom\sfA_\mu$
	and
	$\sfA_\mu u = \lm u$ easily follow. 
	
	\emph{Step 3.} Equality $\s(\sfA_\mu) = \s(\Op)$ combined with~\cite[proof of Thm. 4.8]{AKK16} and~\cite[proof of Thm. 1.3 (d)]{AKK16}
	imply $\s(\Op)\subset\{\lm\in\dC\colon \Re\lm \ge 0\}$ and $\s(\Op)\cap\ii\dR = \{0\}$, respectively.
\end{proof}

\section{Numerical range} \label{sec:numrange}
The goal of this section is to compute the numerical range of $\Op$.
Recall that the numerical range of a linear operator $\sfT$ in a Hilbert space $\cH$
 is the following subset of $\dC$:
\[
\left\{(\sfT\psi,\psi)_{\cH}\colon \psi \in\dom \sfT, \|\psi\|_{\cH} = 1\right\},
\]
which is known to be a convex set.

\begin{thm}\label{thm:numerical_range}
The numerical range of $\Op$ is $\dC$.
\end{thm}

\begin{proof}
Denote by $n$ the inner unit normal on $\partial\Omega$.
For $\eps>0$ denote 
\[
\Omega_\eps=\{x\in\Omega:\, \dist(x,\partial\Omega)<\eps\}.
\]
Recall the well-known fact  (see, \eg, \cite[Sec. 14.6]{GT}):
for sufficiently small $\eps>0$ the distance function
$\rho:\,\overline{\Omega_\eps}\ni x\mapsto \text{dist}\,(x,\partial \Omega)$
is $C^2$-smooth with $|\nabla \rho| = 1$ and $\nabla \rho (x)=n(x)$ for all $x\in\partial\Omega$.

Choose $f\in C^2\big([0,1],\dC\big)$ such that
\[
f(0)= f(1)= f'(1)=0, \quad f'(0)\ne 0,
\]
and define $\varphi_\eps:\Omega\to\dC$ by
\begin{equation*}
 \varphi_\eps(x)=\begin{cases}
             0, & x\in\Omega\setminus\Omega_\eps, \\
             \sqrt{\eps}f \big(\eps^{-1}\rho(x)\big), & x\in\Omega_\eps.
            \end{cases}
\end{equation*}
Denote $b_\eps:=\langle \varphi_\eps\rangle_\mu$.
As $\varphi_\eps$ are uniformly bounded with $\lim_{\eps\to 0}\varphi_\eps(x)=0$ for each $x\in\Omega$,
it follows by the dominated convergence that 
\begin{equation}\label{numbered}
\lim_{\eps\to 0} b_\eps=0.
\end{equation}
Now we pick any $\theta\in C^\infty_0(\Omega)$ with $\langle\theta\rangle_\mu=1$
and denote
\[
\psi_\eps:=\one + \varphi_\eps - b_\eps\theta.
\]
By construction one has $\psi_\eps\in H^2(\Omega)$ and $\langle \psi_\eps\rangle_\mu=1 = \psi_\eps(x)$ for all $x\in\p\Omg$, \ie $\psi_\eps\in\dom\Op$.

Now we will study the asymptotic behaviour of the Rayleigh quotient for $\psi_\eps$ as $\eps$ tends to $0$.
Using the triangle inequality we estimate
\[
	\Big|\|\psi_\eps\|_{L^2(\Omega)}-\|\one\|_{L^2(\Omega)}\Big| 
		\le \|\varphi_\eps - b_\eps \theta\|_{L^2(\Omega)}\le \|\varphi_\eps\|_{L^2(\Omega)}+|b_\eps|\, \|\theta\|_{L^2(\Omega)}.
\]
With the help of~\eqref{numbered} and
\[
\|\varphi_\eps\|^2_{L^2(\Omega)}
\le \|\varphi_\eps\|_\infty^2|\Omg| =
O(\eps)=o(1)
\]
as $\eps \to 0$,
we get
\begin{equation}
  \label{norm4}
\|\psi_\eps\|^2_{L^2(\Omega)}=|\Omega|+o(1).
\end{equation}
Now we use the integration by parts to obtain
\begin{equation}
   \label{green1}
	\begin{aligned}
(\Op\psi_\eps,\psi_\eps)_{L^2(\Omega)} &=\int_{\Omega}(-\nabla\cdot a\nabla \psi_\eps)\overline{\psi_\eps}\dd x\\
&=\int_{\partial\Omega}(n\cdot a\nabla \psi_\eps)\overline{\psi_\eps}\dd \sigma
+\int_{\Omega} a\nabla\psi_\eps\cdot \nabla\overline{\psi_\eps}\dd x. 
\end{aligned}
\end{equation}
For $x\in \partial\Omega$ one has
\[
\nabla \psi_\eps (x)=\nabla \varphi_\eps (x)=\eps^{-1/2}f'(0) \nabla \rho(x)=\eps^{-1/2}f'(0)n(x),
\quad \psi_\eps(x)=1,
\]
hence,
\begin{equation}
   \label{green2}
\int_{\p\Omg}(n\cdot a\nabla \psi_\eps)\overline{\psi_\eps}\dd\sigma
=\beta f'(0) \eps^{-1/2}
\quad \text{ with }
\beta:=\int_{\partial\Omega} n\cdot a n\dd\sigma>0
\end{equation}
(the strict positivity of $\beta$ follows from~\eqref{axi}).
In addition, for a suitable $C>0$ one estimates
$|a\xi\cdot \overline{\xi} |\le C|\xi|^2$ for all $\xi\in\dC^d$,
which gives, as $\eps>0$ is sufficiently small (such that $\supp\theta \cap \Omega_\eps=\emptyset$),
\begin{align*}
\Big|\int_{\Omega} a\nabla\psi_\eps\cdot \nabla\ov{\psi_\eps}\dd x\Big|&\le  C \int_\Omega |\nabla\psi_\eps|^2\dd x\\
&\le C \Big(|b_\eps|^2\int_{\Omega\setminus\Omega_\eps} |\nabla\theta|^{2}\dd x + \int_{\Omega_\eps} |\nabla \varphi_\eps|^2\dd x\Big).
\end{align*}
For $x\in\Omega_\eps$ one estimates
\[
\big|\nabla\varphi_\eps(x)\big|
=\Big|\eps^{-1/2} f'\big(\eps^{-1}\rho(x)\big)\nabla\rho(x)\Big|\le c_1 \eps^{-1/2}, \quad c_1:=\|f'\|_\infty,
\]
and due to $|\Omega_\eps|=O(\eps)$ one obtains $\displaystyle\int_{\Omega_\eps} |\nabla \varphi_\eps|^2\dd x=O(1)$.
Together with $b_\eps=o(1)$ this gives
\begin{equation}
  \label{green3}
\int_{\Omega} a\nabla\psi_\eps\cdot \nabla\overline{\psi_\eps}
=O(1)
\end{equation}
as $\eps \to 0$.

Using \eqref{green2} and \eqref{green3} in \eqref{green1} gives $(\Op\psi_\eps,\psi_\eps)_{L^2(\Omega)}=\beta f'(0) \eps^{-1/2}+O(1)$. By combining
with \eqref{norm4} we arrive at
\begin{align*}
\dfrac{(\Op\psi_\eps,\psi_\eps)_{L^2(\Omega)}}{\|\psi_\eps\|^2_{L^2(\Omega)}}&=\dfrac{\beta f'(0) \eps^{-1/2}+O(1)}{|\Omg|+o(1)}\\
&=\dfrac{\beta}{|\Omg|} f'(0) \eps^{-1/2}+O(1)
\quad\mbox{as }\eps\to0, 
\quad \beta>0.
\end{align*}
By taking $f$ with $f'(0)\in\{1,-1,\ii,-\ii\}$ we conclude that there exists $r_0>0$ such that for any $R_0>0$
one can find $R>R_0$ and $r\in(-r_0,r_0)$ such that the numerical range contains the points $\pm R+ \ii r$ and $\pm \ii R+r$.
As the convex hull of these points is also contained in the numerical range, we conclude that the numerical range covers
the whole complex plane.
\end{proof}

\begin{remark}
As a consequence of Theorem~\ref{thm:numerical_range}, 
the operator $\Op$ is neither m-sectorial nor m-accretive. 
In particular, a similar argument disproves 
the wrong statement of \cite{KK16} that~$\Op$ is m-accretive 
in the case of deterministic measures in one dimension;
see~\cite{KK16c} for a correction.
\end{remark}
\section{Measures with $L^2$-densities}\label{sec:L2}
In this section we consider a special class of absolutely continuous measures on $\Omega$
with square-integrable densities.
For this class of measures the analysis of $\Op$ simplifies. In particular, we are able to compute
the adjoint operator $\Op^*$ and analyse its basic properties. 
\begin{hyp}\label{hyp:measure}
The measure $\mu $ is absolutely continuous with an $L^2$-density, i.e. $\dd\mu(x)=w(x)\dd x$ 
with $w\in L^2(\Omega)$ satisfying $w\geq 0$ and $(\one,w)_{L^2(\Omega)}=1$.
\end{hyp}

Note that under this hypothesis we may write 
$$
  \sfm_\mu(\lambda)= (\RD(\lambda)\one,w)_{L^2(\Omg)}
  \quad \mbox{and} \quad
  \sfm_\mu(\ov\lambda)=(\one,\RD(\lambda) w)_{L^2(\Omg)}
  .
$$

Firstly, we will need the following observation.
\begin{lem} \label{lem:adjoint}
Let $B$ be a bounded everywhere defined linear operator on a Hilbert space $(\cH, (\cdot,\cdot)_\cH)$ and $f,g\in\cH$. Then 
\[
	\big((B\cdot, f)_\cH g\big)^*=( \cdot,g)_\cH B^* f.
\]	
\end{lem}
\begin{proof}
For any $\psi,\varphi \in \cH$ one gets	
\[
	\big( ( B\varphi,f)_\cH g, \psi\big)_\cH= ( B\varphi,f)_\cH(g,\psi)_\cH= \ov{(\psi,g)_\cH}
	(\varphi, B^*f) = 
	(\varphi, (\psi,g)_\cH B^*f).\qedhere
\]
\end{proof}
With this result in hand we will compute the adjoint of $\Op$.
\begin{prop}
 Let the measure $\mu$ be as in Hypothesis~\ref{hyp:measure}. 
 Then 
 \begin{equation*}
  \Op^*=\OpD-( \OpD\cdot,\one)_{L^2(\Omg)} w,\qquad \dom \Op^* = \dom\OpD.
 \end{equation*}
 Moreover, for all $\lm\in\cS_\mu$, 
 \begin{equation} \label{eq:adjointResForm}
  (\Op^*-\lambda)^{-1}=\RD(\lambda)-
   \big(\cdot, \ov\lambda\RD(\ov\lambda)\one+\one\big)_{L^2(\Omg)} \frac{\RD(\lambda)w}{\lambda \sfm_\mu(\lambda)}.
 \end{equation}
\end{prop}

\begin{proof}
First, we will find the adjoint of $(\Op-\ov\lambda)^{-1}$, see \eqref{eq:resolvent}, using Lemma \ref{lem:adjoint} with $B=\RD(\ov\lambda)/(\ov\lambda \sfm_\mu(\ov\lambda))$, $g=\ov\lambda\RD(\ov\lambda)\one+\one$ and $f = w$.  We get
\begin{equation*}
	\big((\Op-\ov\lambda)^{-1}\big)^*
	=
	\RD(\lambda)
	-\big( \cdot,\ov\lambda\RD(\ov\lambda)\one+\one\big)_{L^2(\Omg)}
			\frac{\RD(\lambda)w}{\lambda \sfm_\mu(\lambda)}.
\end{equation*}
Recall that due to Proposition \ref{prop:density} the operator $\Op - \ov\lm$ is densely defined.
Hence, using \cite[Thm. III.5.30]{Kato} we get $\big((\Op-\ov\lambda)^{-1}\big)^* = (\Op^*-\lm)^{-1}$
and the identity \eqref{eq:adjointResForm} follows.

Second, by \eqref{eq:adjointResForm},  $\dom\Op^*=\ran\big((\Op^*-\lambda)^{-1}\big)\subset\dom\OpD$. We will show the other inclusion. Take any $g\in\dom\OpD$ and look for a solution
$u\in L^2(\Omega)$ of the following problem, 
\begin{equation} \label{eq:u_prob}
 g=\RD(\lambda)u-(u, \ov\lambda\RD(\ov\lambda)\one+\one)_{L^2(\Omg)}\frac{\RD(\lambda)w}{\lambda \sfm_\mu(\lambda)} ,
\end{equation}
or, equivalently,
\begin{equation} \label{eq:inv}
 (\OpD-\lambda)g=u-(u, \ov\lambda\RD(\ov\lambda)\one+\one)_{L^2(\Omg)}\frac{w}{\lambda \sfm_\mu(\lambda)}.
\end{equation}
Taking inner product of the latter equation with 
$(\ov\lambda\RD(\ov\lambda)\one+\one)$ we obtain
\begin{equation*}
 ((\OpD-\lambda)g,  \ov\lambda\RD(\ov\lambda)\one+\one)_{L^2(\Omg)}=(u, \ov\lambda\RD(\ov\lambda)\one+\one)_{L^2(\Omg)}\left(1-\frac{(w, \ov\lambda\RD(\ov\lambda)\one+\one )_{L^2(\Omg)}}{\lambda \sfm_\mu(\lambda)}\right),
\end{equation*}
which simplifies to
\begin{equation*}
 (\OpD g, \one)_{L^2(\Omg)}=\frac{-(u, \ov\lambda\RD(\ov\lambda)\one+\one)_{L^2(\Omg)}}{\lambda \sfm_\mu(\lambda)} .
\end{equation*}
Inserting this into \eqref{eq:inv} we  get
\begin{equation} \label{eq:u_sol}
u=(\OpD-\lambda)g-(\OpD g,\one)_{L^2(\Omg)} w.
\end{equation}
The right-hand side of \eqref{eq:u_sol} belongs to $L^2(\Omega)$ and, as one verifies by a direct calculation, it solves \eqref{eq:u_prob}.
Therefore, $g\in\ran\big((\Op^*-\lambda)^{-1}\big)=\dom \Op^*$ and  $u=(\Op^*-\lambda)g$. Putting this together with~\eqref{eq:u_sol} we conclude that $\Op^*g=\OpD g-(\OpD g,\one)_{L^2(\Omg)} w$.
\end{proof}

Recall that $\psi\in\cH$ is a \emph{root vector} for a closed linear operator $\sfH$ in a Hilbert space $\cH$ if there exists $\lm\in\dC$ and $n\in\dN$ such that $(\sfH - \lm)^n\psi = 0$. Recall also that the completeness of a family of vectors $\{\psi_j\}_{j\in\dN}$ in a Hilbert space $\cH$  means that its span is
	dense in $\cH$, or equivalently, 
	$(\{\psi_j\}_{j\in\dN})^\bot = \{0\}$.

\begin{prop}\label{prop:adjoint} Let $\mu$ be a measure as in Hypothesis~\ref{hyp:measure}. Then
\begin{myenum}
 \item  $\Op^*$ has compact resolvent.
 \item $\ker\Op^*=\spn\{\OpD^{-1}w\}$. 
 \item 
If $(\lambda,g)$ is an eigenpair of $\OpD$ such that $(g,\one)_{L^2(\Omg)}=0$ then $(\lambda,g)$ is also an eigenpair of $\Op^*$. If $(\lambda,g)$ is an eigenpair of $\Op^*$ such that $(\OpD g,\one)_{L^2(\Omg)}=0$ then $(\lambda,g)$ is an eigenpair of $\OpD$.
 \item For any eigenfunction $g$ of $\Op^*$ corresponding to a non-zero eigenvalue,
one has $( \one,g)_{L^2(\Omg)}=0$.
 \item $\lambda\in\sigma(\Op^*)\setminus(\{0\}\cup\sigma(\OpD))$ 
if, and only if, $\sfm_{\mu}(\lambda)=0$. In the positive case, $\lambda$ is a geometrically simple eigenvalue and $\ker(\Op^*-\lambda)=\spn\{(\OpD-\lambda)^{-1}w\}$. If $\lambda\in\sigma(\Op^*)$ then $\ov\lambda\in\sigma(\Op^*)$.
 \item $\sigma(\Op)=\sigma(\Op^*)$ and the geometric multiplicities of the eigenvalues coincide.
 \item The families of root vectors of $\sfH_\mu^*$ and $\sfH_\mu$
 	are complete in the Hilbert space $L^2(\Omega)$. In particular, the spectrum of $\sfH_\mu$ consists of an infinite number of isolated
 	points that accumulate at complex infinity.
\end{myenum} 
\end{prop}
\begin{proof}
 The claim of~(i) follows either directly from \eqref{eq:adjointResForm} or from the fact that $\Op$ has compact resolvent. The eigenvalue equation reads
 \begin{equation} \label{eq:AdjointEvEq}
  \OpD g-(\OpD g,\one)_{L^2(\Omg)} w=\lambda g.
 \end{equation}
 If $\lambda=0$ then $g=\OpD^{-1}w\neq 0$ is the only (up to a multiplicative constant) solution of the above equation. Hence, we get (ii). Under the assumption of the first statement in (iii), $(\OpD g,\one)_{L^2(\Omg)}=0$. Consequently, the eigenvalue equation for $\Op^*$ coincides with that for $\OpD$. Similarly, we get the second statement of (iii). Multiplying \eqref{eq:AdjointEvEq} by $\one$ from the right we arrive at (iv).
 
Let $\lambda\notin \{0\}\cup\sigma(\OpD)$. Then $\RD(\lambda)w\neq 0$ and $\ov\lambda\RD(\ov\lambda)\one+\one\neq 0$. The latter follows from the fact that $(\RD(\ov\lambda)\one)|_{\p\Omg}=0$. Therefore, the zeros of $\sfm_\mu(\lambda)$ are indeed singularities of $(\Op^*-\lambda)^{-1}$, cf. \eqref{eq:adjointResForm}. The corresponding eigenfunction is given uniquely (up to a multiplicative constant) by inverting in \eqref{eq:AdjointEvEq}.  To get the last statement of (v) we observe that  $\sfm_\mu(\ov\lambda)=\overline{\sfm_\mu(\lambda)}$. 

Finally, let $\dim\ker(\Op-\lambda)=k\in\N$. Then for any $\nu\in\rho(\Op)$, $\dim\ker((\Op-\nu)^{-1}-\frac{1}{\lambda-\nu})=k$. By the Fredholm alternative, 
$\dim\ker((\Op^*-\ov\nu)^{-1}-\frac{1}{\ov\lambda-\ov\nu})=k$ which in turn yields $\dim\ker(\Op^*-\ov\lambda)=k$. Since by Proposition~\ref{prop:conj_ev} spectra of $\Op$ and $\Op^*$ are both invariant with respect to the complex conjugation and the geometric multiplicities of conjugate eigenvalues coincide, we get (vi).

In the following we denote by $\sfI$ the identity operator in $L^2(\Omg)$.
Using~\eqref{eq:adjointResForm} the inverse of $\Op^* +\sfI$ can be expressed as
\begin{equation}\label{eq:factorization}
\begin{aligned}
	(\Op^* + \sfI)^{-1} & = \RD(-1)
	+ \frac{\RD(-1)w}{\sfm_\mu(-1)}
	(\cdot,\one - \RD(-1)\one)_{L^2(\Omg)}\\
	& =
	\RD(-1)\left(\sfI + 
	\frac{w}{\sfm_\mu(-1)}
	(\cdot,\one - \RD(-1)\one)_{L^2(\Omg)}\right).
\end{aligned}
\end{equation}
Note that $\ker((\sfH_\mu^* + \sfI)^{-1}) = \{0\}$. Moreover, in view of the Weyl's eigenvalue asymptotics,
\[
	\sum_{k=1}^\infty \frac{1}{(\lm_k + 1)^p} <\infty,\qquad \forall\,p > \frac{d}{2},
\]
\ie, the self-adjoint operator $\RD(-1)$
belongs to the Schatten--von Neumann class
$\sS_p(L^2(\Omg))$ for all $p  > \frac{d}{2}$. On the other hand, the finite rank projector 
$\frac{w}{\sfm_\mu(-1)}
(\cdot,\one - \RD(-1)\one)_{L^2(\Omg)}$ is a compact operator.
Hence, combining the representation~\eqref{eq:factorization} with~\cite[Chap. V, Thm. 8.1]{GK69} we obtain that
the system of root vectors of $(\Op^* + \sfI)^{-1}$ is complete. The same is true for $\Op^*$ since it has the same
system of root vectors; \cf~\cite[Thm. IX.2.3 and its proof]{EE}.
Applying~\cite[Chap. V, Rem. 8.1]{GK69} we obtain that the system of root vectors of $(\Op + \sfI)^{-1}$ and hence of $\Op$ is complete as well. Completeness of the system of root vectors implies that the spectrum of $\Op$ consists of infinitely many eigenvalues. Taking discreteness of the spectrum of $\Op$ into account, we conclude that these eigenvalues accumulate at complex infinity. Thus, the proof of (vii) is complete. 
\end{proof}

\section{Spectral enclosures for non-real eigenvalues}\label{sec:enclosures}
In this section we will consider perturbations of two special examples of measures $\mu$, namely of the uniform measure and then 
of $\dd\mu(x) = \frac{\chi_1(x)}{(\one,\chi_1)_{L^2(\Omg)}}\dd x$, 
where $\chi_1$ is the ground-state of $\OpD$. In both unperturbed cases, $\sigma(\Op)\subset\dR$. We will show, in particular, that for perturbations of the measures obeying some smallness conditions this remains true at least for spectra in half-planes $\Re \lm \le a$ with some $a>\lm_1$. The following observation will be useful.
\begin{lem} \label{lem:nec_cond}
	Let $\mu$ be as in Hypothesis~\ref{hyp:measure}. If $\lm\in\dC\setminus\dR$ is an eigenvalue of $\Op$ then
	\begin{equation} \label{eq:cond}
	 \big(\RD(\lm)\RD(\ov\lm)\one,\,w\big)_{L^2(\Omg)} = 0.
	\end{equation}
\end{lem}
\begin{proof}
Recall that $\lm\in\rho(\OpD)\sm\{0\}$ is an eigenvalue of $\Op$ 
if, and only if, the condition 
\[
	(\RD(\lm) \one, w)_{L^2(\Omg)} = 0
\]
holds. If $\lm\in\dC\sm\dR$ is an eigenvalue of $\sfH_\mu$, then, by Proposition~\ref{prop:conj_ev}, $\ov\lm$ is one as well and we get
\[
	\big((\RD(\lm) - \RD(\ov\lm))\one, w\big)_{L^2(\Omg)} = 0.
\]
Using the resolvent identity we eventually obtain 
\begin{equation*}
	\big(\RD(\lm)\RD(\ov\lm)\one,
							w\big)_{L^2(\Omg)} = 0.\qedhere
\end{equation*}
\end{proof}

\subsection{Perturbation of the uniform probability measure}
\label{ssec:uniform}
In this subsection we consider the uniform measure 
\[
	\boxed{\dd\mu_0 = \frac{\dd x}{|\Omega|}}
\]
as the unperturbed measure. Then
\begin{equation*}
	\sfm_{\mu_0}(\lm)=\frac{1}{|\Omega|}\sum_{n=1}^{+\infty}\frac{(\chi_n,\one)_{L^2(\Omg)}^2 }{\lambda_n-\lambda}.
\end{equation*}

Taking the imaginary part of the equation $\sfm_{\mu_0}(\lm)=0$ we infer that $\Im\lambda=0$. Therefore, $\sigma(\sfH_{\mu_0})\subset\dR$.
Recall that by Proposition~\ref{prop:0_lm1} there is no non-zero eigenvalue of $\sfH_{\mu_0}$ below the lowest eigenvalue $\lm_1$
of $\OpD$.
The inclusion $\s(\sfH_{\mu_0})\subset\dR$ can alternatively be shown through an algebraic argument with an eigenfunction. Let $u$ be an eigenfunction of $\sfH_{\mu_0}$ corresponding to $\lm \ne 0$. Clearly, $u$ is not a constant function.
Then we have
\[
	-\nabla \cdot a\,\nabla  u = \lm u\and \frac{1}{|\Omg|}\int_\Omg u(x)\dd x = u|_{\p\Omg}.
\]
Hence, we get integrating by parts
\[
\begin{aligned}
	\lm\int_\Omg |u|^2\dd x & = \int_\Omg(-\nabla \cdot a\, \nabla u)\ov{u}\dd x = \int_\Omg a\nabla u\cdot \nabla \ov{ u} \dd x +
	\int_{\p\Omg}n \cdot a\nabla u \ov{u} \dd \s \\
	& =
	\int_\Omg a\nabla u\cdot \nabla \ov{ u}\dd x +
		\frac{1}{|\Omg|}\ov{\int_\Omg u\dd x} 
	\int_{\p\Omg}n \cdot a\nabla u  \dd \s\\
	& =
	\int_\Omg a\nabla u\cdot \nabla\ov{ u}\dd x +
	\frac{1}{|\Omg|}\ov{\int_\Omg u\dd x} 
	\int_{\Omg}(-\nabla \cdot a\, \nabla u)\dd x\\
	& =
	\int_\Omg a\nabla u\cdot\nabla \ov{ u}\dd x +
	\frac{\lm}{|\Omg|}\ov{\int_\Omg u\dd x} 
	\int_{\Omg}u\dd x\\
	&=
	\int_\Omg a\nabla u\cdot\nabla \ov{ u}\dd x
	+
	\frac{\lm}{|\Omg|}\left|\int_\Omg u\dd x\right|^2;
\end{aligned}	 
\]
here $n$ stands for the inner unit normal. 
Finally, $\lm$ can be expressed as 
\begin{equation}\label{eq:lm_algebraic}
	\lm = \frac{\displaystyle\int_\Omg a\nabla u\cdot \nabla \ov{ u}\dd x}{\displaystyle \int_\Omg|u|^2\dd x - \frac{1}{|\Omg|}\left|\int_\Omg u \dd x\right|^2}\in\dR,
\end{equation}
where the denominator is not equal to zero by 
the Cauchy--Schwarz inequality, taking that $u$ is not a 
constant function into account. The representation~\eqref{eq:lm_algebraic} yields as a by-product that
all the eigenvalues of $\sfH_{\mu_0}$ are real. This representation is also of certain interest in its own right.

In the following hypothesis we introduce a class of measures.
\begin{hyp}\label{hyp:measure2}
The measure $\mu$ has the form $\dd\mu(x)=\left(\dfrac{\one}{|\Omega|}+v(x)\right)\dd x$ with
\[
v\in L^2(\Omg), \quad (v,\one)_{L^2(\Omg)} = 0, \quad v \ge -\frac{1}{|\Omega|}~\text{ point-wise,}
\]
and for some $k\in\dN$ one has
	\begin{equation}\label{eq:smallness}
	\|v\|_{L^2(\Omg)} < \frac{1}{|\Omega|}\min_{n\in\{1,\dots,k\}}
	\left\{|(\chi_n,\one)_{L^2(\Omg)}|\right\}.
	\end{equation}
\end{hyp}	
\begin{remark}
	For a generic domain the eigenfunctions $\chi_n$ are not orthogonal to the constant function and the Hypothesis~\ref{hyp:measure2} can easily be met. However, for more special domains with symmetries such as a ball or a (hyper)cube one has $(\chi_2, \one) = 0$ and the condition~\eqref{eq:smallness} can only be satisfied for $k =1$.
\end{remark}	
Recall the definition
\begin{equation*}
	\mathbb{H}_k 
	= \left\{\lm\in\dC\setminus\dR\colon \Re \lm \le \frac{\lm_k+ \lm_{k+1}}{2}\right\}.
\end{equation*}

\begin{thm}\label{thm1}
	Under Hypothesis~\ref{hyp:measure2} there holds $\s(\Op) \cap \dH_k = \emptyset$.
\end{thm}	
\begin{proof}
	We prove the claim by contradiction. Suppose that $\lm\in\dH_k$ is an eigenvalue of $\Op$.
	The condition~\eqref{eq:cond} takes the form 
	\[
	\frac{1}{|\Omega|}\|\RD(\lm)\one\|^2_{L^2(\Omg)} 
	+ 
	(\RD(\lm)\RD(\ov\lm)\one,v)_{L^2(\Omg)} = 0,
	\]
	and the Cauchy--Schwarz inequality then implies that
	\[
	\|\RD(\lm)v\|_{L^2(\Omg)}
	\ge \frac{\|\RD(\lm)\one\|_{L^2(\Omg)}}{|\Omega|}. 
	\] 
	Note further that 
	\[
	\begin{aligned}
	\|\RD(\lm)\one\|^2_{L^2(\Omg)} 
	&= 
	\sum_{n=1}^\infty
	\frac{|(\chi_n,\one)_{L^2(\Omg)}|^2}{|\lm - \lm_n|^2} 
	\ge 
	\sum_{n=1}^{k}
	\frac{|(\chi_n,\one)_{L^2(\Omg)}|^2}{|\lm - \lm_n|^2} \\
	&\ge
	\min_{n\in\{1,\dots,k\}}
	\left\{|(\chi_n,\one)_{L^2(\Omg)}|^2\right\}
	\left(\sum_{n=1}^k \frac{1}{|\lm -\lm_n|^2}\right)\\
	&\ge
	\min_{n\in\{1,\dots,k\}}
	\left\{|(\chi_n,\one)_{L^2(\Omg)}|^2\right\}\cdot
	\max_{n\in\{1,\dots,k\}}
	\left\{
	\frac{1}{|\lm-\lm_n|^2}\right\}.
	\end{aligned}
	\]
	On the other hand, due to $\lm\in\dH_k$, one has
	\[
	\begin{aligned}
	\|\RD(\lm) v\|_{L^2(\Omg)}^2 
	&=
	\sum_{n=1}^\infty
	\frac{|(\chi_n,v)_{L^2(\Omg)}|^2}{|\lm - \lm_n|^2}\\
	& \le 
	\left(\sum_{n=1}^\infty
		|(\chi_n,v)_{L^2(\Omg)}|^2\right)
	\left(\max_{n\in\dN}\left\{
	\frac{1}{|\lm - \lm_n|^2}\right\}\right)\\
	&\le 
	\|v\|_{L^2(\Omg)}^2\max_{n\in\{1,\dots,k\}} 
	\left\{
	\frac{1}{|\lm-\lm_n|^2}\right\}, 
	\end{aligned}
	\]
	where we applied the Parseval identity 
and used the fact that the closest eigenvalue of $\sfH_{\rm D}$ to $\lm$ is among $\{\lm_n\}_{n=1}^k$.
	Hence, the assumption~\eqref{eq:smallness} 
	implies
	\[
		\|\RD(\lm)v\|_{L^2(\Omg)}
		< \frac{\|\RD(\lm)\one\|_{L^2(\Omg)}}{|\Omega|} 	
	\] 
	and we get
	a contradiction.
\end{proof}	
In the next theorem we characterise low lying real eigenvalues of $\Op$ under the same assumption as in Theorem~\ref{thm1}. We obtain an interlacing property. 
\begin{thm}\label{thm2} 
	Let $\mu$ be a measure as in Hypothesis~\ref{hyp:measure2}.
	Assume that $\Omega$ is such that  the low eigenvalues $\{\lm_n\}_{n\ge 2}^{k}$ of $\OpD$ are simple. 
	Then there is a unique eigenvalue of $\Op$ in the interval $(\lm_i,\lm_{i+1})$
	for each $i \in\{1,2,\dots,k-1\}$, and its geometric multiplicity is one.
\end{thm}	

\begin{figure} 
\centering
\includegraphics[width=13cm]{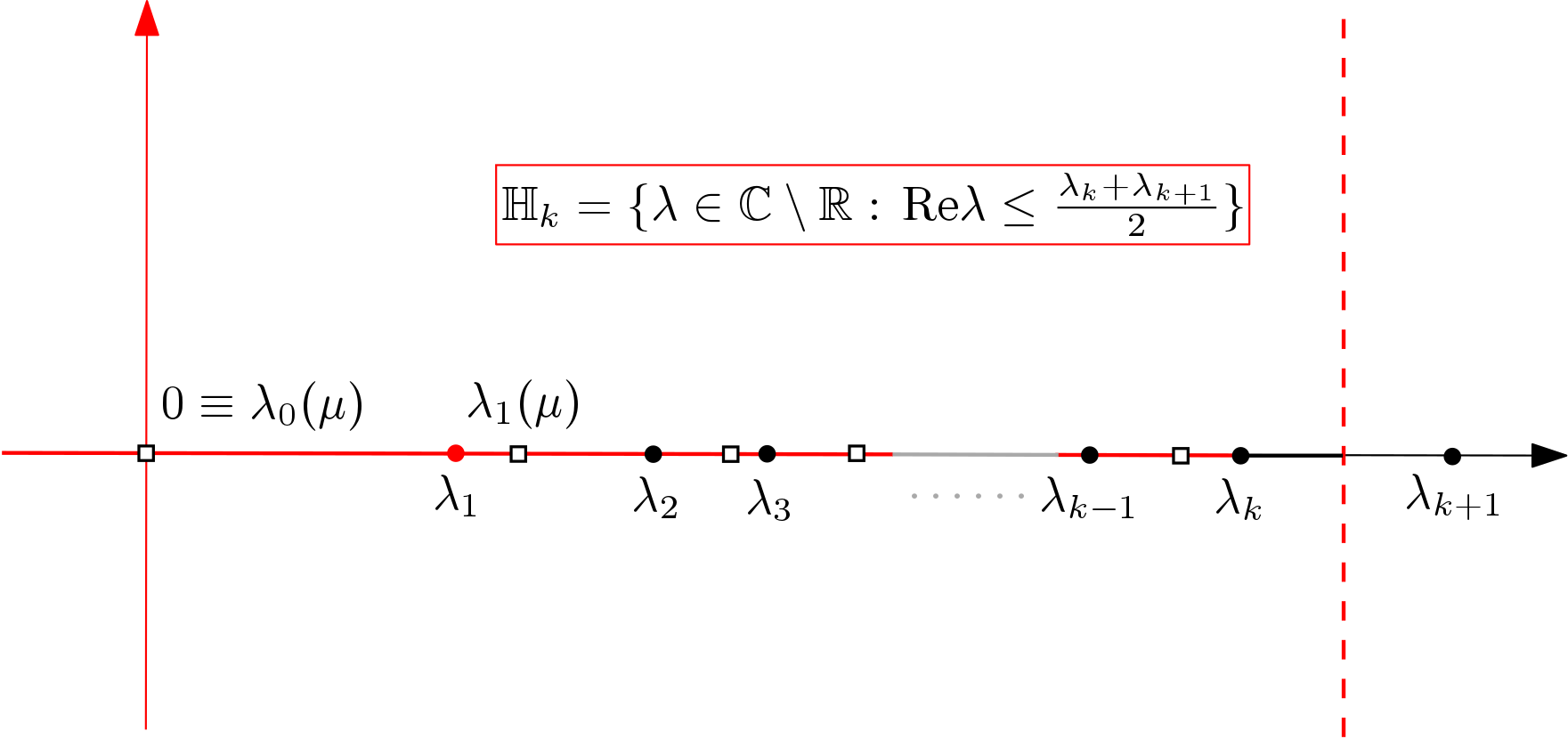}  
\caption{Spectral enclosure given by Theorem \ref{thm1}. The only eigenvalues of $\Op$ in the cut half-plane $\dH_k$  are either those localized in Theorem~\ref{thm2} (depicted as white rectangles) or possibly the higher Dirichlet eigenvalues (depicted as black discs) and some points on the black thick line segment. }
\end{figure} 
 
\begin{proof}
   	First,
	consider for $n\in\{1,2,\dots,k\}$ the coefficients 
	\[
	\begin{aligned}
	\aa_n :=\ &
	\frac{1}{|\Omega|}
	|(\chi_n,\one)_{L^2(\Omg)}|^2
	+
	(\chi_n,v)_{L^2(\Omg)}(\one,\chi_n)_{L^2(\Omg)}\\
	\ge \ &
	\frac{1}{|\Omega|}|(\chi_n,\one)_{L^2(\Omg)}|^2
	-
	\|v\|_{L^2(\Omg)} |(\one,\chi_n)_{L^2(\Omg)}|\\
	= \ &
	|(\chi_n,\one)_{L^2(\Omg)}|
	\left(
	\frac{1}{|\Omega|}
	|(\chi_n,\one)_{L^2(\Omg)}|
	-
	\|v\|_{L^2(\Omg)}\right) > 0,
	\end{aligned}	
	\]	
	where the last inequality follows from~\eqref{eq:smallness}.
	Let $i\in \{1,2,\dots,k-1\}$ be fixed
	and consider the real-valued continuous function
	\[
	(\lm_i,\lm_{i+1})\ni \lm\mapsto f(\lm) := \sfm_\mu(\lm).
	\]
	It can be decomposed as
	\[
	f(\lm) = 
	\frac{\aa_i}{\lm_i  -\lm}
	+
	\frac{\aa_{i+1}}{\lm_{i+1} - \lm}
	+
	g(\lm),
	\]
	where the function 
$\lambda \mapsto g(\lambda)$ 
has finite limits
	as $\lm$ tends to the endpoints of the interval
	$(\lm_i,\lm_{i+1})$.
	Hence, we get
	\begin{equation}\label{eq:limits}
	\lim_{\lm\arr\lm_i^+} f(\lm) = -\infty\and
	\lim_{\lm\arr\lm_{i+1}^-} f(\lm) = +\infty.
	\end{equation}
	Now, let us look at the derivative of $f$,
	\[
	\begin{aligned}
	f'(\lm) & = \left(\RD(\lm)^2\one,\tfrac{\one}{|\Omega|}+v\right)_{L^2(\Omg)}\\
	&\ge
	\frac{1}{|\Omega|}
	\|\RD(\lm)\one\|_{L^2(\Omg)}^2
	-\|\RD(\lm)\one\|_{L^2(\Omg)}
	\|\RD(\lm)v\|_{L^2(\Omg)}\\
	& =
	\|\RD(\lm)\one\|_{L^2(\Omg)}
	\left(\frac{1}{|\Omega|}\|\RD(\lm)\one\|_{L^2(\Omg)}
	-		\|\RD(\lm)v\|_{L^2(\Omg)}\right) > 0 ,
	\end{aligned}
	\]
	where the inequality $\frac{1}{|\Omega|}\|\RD(\lm)\one\|_{L^2(\Omg)}
	> \|\RD(\lm)v\|_{L^2(\Omg)}$ follows in exactly the same way as in the proof of Theorem~\ref{thm1}.
	
	Combining Theorem~\ref{thm0}\,(iii) with continuity of $f$, positivity of $f'$ and the limits~\eqref{eq:limits} we get that there is exactly one eigenvalue of $\Op$ in the interval $(\lm_i,\lm_{i+1})$. 
Geometric multiplicity of these eigenvalues is one in view of Theorem~\ref{thm0}\,(iii).	
\end{proof}	
\begin{remark}
As seen from the proof, some generalizations are possible. For example, one can include the case of multiple eigenvalues by making the single assumption $\lambda_{k}<\lambda_{k+1}$. Then there will be an eigenvalue in each interval $(\lambda_{j},\lambda_{j+1})$ with $\lambda_{j}<\lambda_{j+1}\leq\lambda_{k}$.
\end{remark}

\begin{prop}\label{prop:real}
	Let $\mu$ be a measure as in Hypothesis~\ref{hyp:measure2} satisfying~\eqref{eq:smallness}
	with $k= 2$.
	Let the second eigenvalue $\lm_2$ of $\OpD$ be simple.
	Then the point $\lm_1(\mu)$ is a real eigenvalue of $\Op$, 
it belongs to the interval
	$(\lm_1,\lm_2)$ and satisfies
	the  bound
	\[
	\left|\frac{\aa_1}{\lm_1 - \lm_1(\mu)}
	+
	\frac{\aa_2}{\lm_2 - \lm_1(\mu)}\right|
	\le \frac{|\Omg|^{1/2}\|w\|_{L^2(\Omg)}}{\lm_3 - \lm_2},
	\]
	where $\aa_j = (\one,\chi_j)_{L^2(\Omg)}(\chi_j,w)_{L^2(\Omg)}$, $j=1,2$, and $w = \frac{\one}{|\Omg|} +v$.
\end{prop}	
\begin{proof}
	By Theorem~\ref{thm1} there are no eigenvalues of $\Op$ in the set $\mathbb{H}_2$. On the other hand, 
	by Theorem~\ref{thm2} there is a real eigenvalue of $\Op$ in the interval $(\lm_1,\lm_2)$.
	Hence, the non-zero eigenvalue of $\Op$ with the smallest real part is real and belongs to the interval $(-\infty,\lm_2)$.
	By Proposition~\ref{prop:0_lm1} there are no eigenvalues of $\Op$ in the interval $(-\infty,\lm_1)$. According to~\cite[Thm. 5]{BP07} $\lm_1$ is not an eigenvalue of $\Op$. Hence, we conclude that $\lm_1(\mu) \in (\lm_1,\lm_2)$ is an eigenvalue of $\Op$.
	
	Let us decompose the function $\sfm_\mu(\lm)$ as 
	\[
		\sfm_\mu(\lm)		= 
		\frac{\aa_1}{\lm_1 - \lm}
		+
		\frac{\aa_2}{\lm_2 - \lm}
	 +
		\sum_{n \ge 3} \frac{(\one,\chi_n)_{L^2(\Omg)}(\chi_n,w)_{L^2(\Omg)}}{\lm_n - \lm}.
	\]
	Now, we introduce the comparison function
	\[
		g(\lm) := \frac{\aa_1}{\lm_1 - \lm}
		+
		\frac{\aa_2}{\lm_2 - \lm}.
	\]
	The difference between $\sfm_\mu(\lm)$ and $g(\lm)$ for $\lm \in (\lm_1,\lm_2)$ can be estimated using the Cauchy--Schwarz inequality as follows
	\[
		|\sfm_\mu(\lm) - g(\lm)| 
		\le 
		\frac{1}{\lm_3-\lm_2} 
		\sum_{n\ge 3}|(\one,\chi_n)_{L^2(\Omg)}|\cdot
		|(\chi_n,w)_{L^2(\Omg)}|
		\le 
		\frac{|\Omega|^{1/2}\|w\|_{L^2(\Omg)}}{\lm_3-\lm_2} .
	\]
	Substituting $\lm = \lm_1(\mu)$ and using that $\sfm_\mu(\lm_1(\mu)) = 0$ we end up with
	\[
		|g(\lm_1(\mu))| \le \frac{|\Omega|^{1/2}\|w\|_{L^2(\Omg)}}{\lm_3-\lm_2},
	\]
	which is the desired inequality.
\end{proof}	

\subsection{Perturbations of the ground-state} 
\label{ssec:chi1}
Let $\chi_1 > 0$ be the $L^2$-normalized ground-state of  
$\OpD$. In this subsection we consider measures which are perturbations of
\[
	\boxed{\dd\mu_1(x)=\frac{\chi_1(x)\dd x}{(\chi_1,\one)_{L^2(\Omega)}}.}
\]
Then
we get 
\begin{equation*}
	\sfm_{\mu_1}(\lm)=\frac{1}{\lm_1-\lm}
\end{equation*}
and \eqref{eq:resolvent} reduces to 
\begin{equation*}
\sfR_{\mu_1}(\lm) 
	:= 
	\sfR_{\rm D}(\lm)-\left(\lm\sfR_{\rm D}(\lm)\one  + \one\right)
\frac{1}{(\chi_1,\one)_{L^2(\Omg)}\lm}(\cdot,\chi_1)_{L^2(\Omg)}. 
\end{equation*}
Moreover, by~\cite[Thm. 1]{BP07} we have $\sigma(\sfH_{\mu_1})=(\sigma(\OpD)\sm \{\lm_1\})\cup\{0\}$.

 Then we have the following enclosure for the non-real spectrum.
\begin{thm}\label{thm3}
	Assume
	\begin{gather*}
		\dd\mu(x)
		= \left(\frac{\chi_1(x)}{(\chi_1,\one)_{L^2(\Omega)}} + v(x)\right)\dd x, \quad v\in L^2(\Omg),\\
	\|v\|_{L^2(\Omg)} < |\Omg|^{-1/2},
	\quad (v,\one)_{L^2(\Omg)}=0, \quad v\geq-\frac{\chi_1}{(\chi_1,\one)_{L^2(\Omega)}} \text{ pointwise.}
	\end{gather*}
	Then
	\[
		\s(\Op)\sm [0,\infty) \subset \left\{\lm\in\dC\sm\dR\colon
		\frac{\dist\big(\lm,\s(\sfH_{\rm D}))}{|\lm_1 - \lm|} \le |\Omg|^{1/4}\|v\|_{L^2(\Omg)}^{1/2}
		\right\},	
	\]
	in particular, $\sigma(\Op)\cap\dH_1=\emptyset$.
\end{thm}
\begin{proof}
 We start with the expansion
\[
	\left(\RD(\lm)\RD(\ov\lm)\one, 
					\frac{\chi_1}{(\chi_1,\one)_{L^2(\Omega)}}+ v\right)_{L^2(\Omg)} 
	= 
	\frac{1}{|\lm_1 - \lm|^2} 
	+ 
	\big(\RD(\lm)\RD(\ov\lm)\one,
							v\big)_{L^2(\Omg)}.
\]
The second term can be estimated using
the Cauchy--Schwarz inequality and the spectral theorem as follows
\[
\begin{aligned}
	\left|\big(\RD(\lm)\RD(\ov\lm)\one, v\big)_{L^2(\Omg)}\right|
	&\le
	\big\|\RD(\lm)\RD(\ov\lm)\one\big\|_{L^2(\Omg)}  \|v\|_{L^2(\Omg)}\\
	& \le
	\big\|\RD(\lm)\RD(\ov\lm)\big\|
	\|\one\|_{L^2(\Omg)} \|v\|_{L^2(\Omg)}\\
	&\le
	\frac{1}{\big(\dist\big(\lm,\s(\sfH_{\rm D}))\big)^2}
	\sqrt{|\Omg|}\|v\|_{L^2(\Omg)}.
\end{aligned}
\]
Hence, taking Lemma \ref{lem:nec_cond} into account, the condition
\begin{equation*}
	\dist\big(\lm,\s(\sfH_{\rm D}))
	\frac{1}{|\lm_1 - \lm|} > |\Omg|^{1/4}\|v\|_{L^2(\Omg)}^{1/2}.
\end{equation*}
ensures that $\lm\in\dC\sm\dR$ is not an eigenvalue of $\sfH_\mu$. 
For $\lm \in \dH_1$ this condition clearly holds in view of $\|v\|_{L^2(\Omg)} < |\Omg|^{-1/2}$.
\end{proof}

\section*{Acknowledgments}
The research was supported by the Czech-French MOBILITY project No. 8J18FR033.
D.K.\ was supported by the GACR grants No.~18-08835S and 20-17749X 
of the Czech Science Foundation. K.P.\ was supported in part by the PHC Amadeus 37853TB funded by the French Ministry of Foreign Affairs and the French Ministry of Higher Education, Research and Innovation.
M.T.\ was supported by the project 
CZ.02.1.01/\-0.0/0.0/\-16\_019/\-0000778 from the European Regional Development Fund. The authors wish to express their thanks to Sergey Denisov for stimulating discussions. The authors are also grateful to the anonymous referee whose suggestions led to improvements in the manuscript.

\bibliographystyle{alpha}

\begin{thebibliography}{\textsc{BCD{\etalchar{+}}72}}

\bibitem[AKK16]{AKK16}
W.~Arendt, S.~Kunkel, and M.~Kunze, Diffusion with nonlocal boundary conditions, {\it J. Funct. Anal.} {\bf 270} (2016), 2483--2507.
	
\bibitem[B14]{B14}
I.~Ben-Ari, Coupling for drifted Brownian motion on an interval with redistribution from the boundary, {\it 	Electron. Commun. Probab.} {\bf 19} (2014), 11p.

\bibitem[BP07]{BP07}
I.~Ben-Ari and R.~Pinsky, Spectral analysis of a family of second-order elliptic operators with nonlocal boundary condition indexed by a probability measure, {\it J. Funct. Anal.} {\bf 251} (2007), 122--140.

\bibitem[BP09]{BP09}
I.~Ben-Ari and R.~Pinsky, Ergodic behavior of diffusions with random jumps from the boundary, {\it Stochastic Processes Appl.} {\bf 119} (2009), 864--881.

\bibitem[Da]{D89}
E.\,B.~Davies, \emph{Heat kernels and spectral theory},
Cambridge University Press, Cambridge, 1989.

\bibitem[EE]{EE}
D.\,E.~Edmunds, W.\,D.~Evans, \emph{Spectral theory and differential operators}, Oxford University Press, Oxford, 2018.

\bibitem[GK]{GK69}
I.\,C.~Gohberg and M.\,G.~Krein, 
\emph{Introduction to the theory of linear nonselfadjoint operators}
American Mathematical Society, Providence, 1969.	

\bibitem[GT]{GT}
D. Gilbarg and N. S. Trudinger,
\emph{Elliptic partial differential equations of second order.}
Grundlehren der mathematischen Wissenschaften. vol.~224 (2nd ed.). Springer-Verlag, 1983. 

\bibitem[Gr]{G85} P.~Grisvard, \emph{Elliptic Problems in Nonsmooth Domains}, Pitman, Boston, 1985.


\bibitem[Ka]{Kato}
T.~Kato, \emph{Perturbation theory for linear operators}, Springer-Verlag,
Berlin, 1995.


\bibitem[KK16]{KK16}
M.~Kolb and D.~Krej\v{c}i\v{r}\'{i}k, 
Spectral analysis of the diffusion operator with random jumps from the boundary, {\it Math. Z.} {\bf 284} (2016), 877--900.

\bibitem[KK19]{KK16c}
M.~Kolb and D.~Krej\v{c}i\v{r}\'{i}k, 
Correction to: Spectral analysis of the diffusion operator with random jumps from the boundary, 
{\it Math. Z.} {\bf 296} (2020), 883--885.

\bibitem[KW11]{KW11}
M.~Kolb and A.~W\"{u}bker, Spectral analysis of diffusions with jump boundary, {\it J. Funct. Anal.} {\bf 261} (2011), 1992-2012.


\bibitem[K20]{K18}
M.~Kunze, Diffusion with nonlocal Dirichlet boundary conditions on domains, 
{\it Stud. Math.} {\bf 253}  (2020), 1--38.
%

\bibitem[LLR08]{LLR}
Y.~Leung, W.~Li, Rakesh,
Spectral analysis of Brownian motion with jump boundary,
{\it Proc. Amer. Math. Soc.} {\bf 136} (2008),  4427--4436.


\bibitem[Y18]{Y18}
J.~Yan, Dependence of eigenvalues on the diffusion operators with random jumps from the boundary, {\it J. Differ. Equations} {\bf 266} (2019), 5532--5565.

\bibitem[YS19]{SY19}
J.~Yan and G.~Shi, 
Multiplicities of eigenvalues of the diffusion operator with random jumps from the boundary, {\it Bull. Aust. Math. Soc.} {\bf 99} (2019), 101--113.

\end{thebibliography}

\newcommand{\etalchar}[1]{$^{#1}$}

\end{document}